\documentclass[11pt]{amsproc}

\usepackage{mathtext}
\usepackage[cp1251]{inputenc}

\usepackage{bm}

\usepackage{amsmath}
\usepackage{amssymb}
\usepackage{amsxtra}

\usepackage[dvips]{graphicx}
\usepackage{amsmath}
\usepackage{amssymb}
\usepackage{amsxtra}

\usepackage{epsfig}
\usepackage{epic}
\usepackage{eepic}
\usepackage{graphics}
\usepackage{graphicx}
\usepackage{subfigure}

\usepackage{caption}
\def\N{{{\Bbb N}}}
\def\Z{{{\Bbb Z}}}
\def\T{{{\Bbb T}}}
\def\R{{\Bbb R}}

\def\l{{\lambda }}

\def\D{{\Delta }}

\def\vp{{\varphi}}

\def\g{{\gamma }}

\def\){\right)}
\def\({\left(}

\def\I{{\mathcal{I} }}

\def\Np{{\mathcal{N} }}

\numberwithin{equation}{section}

\newtheorem{lemma}{Lemma}[section]
\newtheorem{theorem}{Theorem}[section]

\newtheorem{remark}{Remark}[section]

\def\G{{\Gamma }}

\def\n{\bm{n}}
\def\x{\bm{x}}

\par

\begin{document}

\title[Asymptotics of the Lebesgue constant]{Asymptotics of the Lebesgue constants for bivariate approximation processes}

\author[Yurii Kolomoitsev]{Yurii Kolomoitsev$^{\text{a, b, *, 1}}$}
\address{Universit\"at zu L\"ubeck, Institut f\"ur Mathematik,
Ratzeburger Allee 160, 23562 L\"ubeck, Germany}
\email{kolomoitsev@math.uni-luebeck.de}

\author[Tetiana Lomako]{Tetiana Lomako$^{\text{b}}$}
\address{Institute of Applied Mathematics and Mechanics of NAS of Ukraine,
General Batyuk Str. 19, Slov’yans’k, Donetsk region, Ukraine, 84100}
\email{tlomako@yandex.ua}

\thanks{$^\text{a}$Universit\"at zu L\"ubeck,
Institut f\"ur Mathematik,
Ratzeburger Allee 160,
23562 L\"ubeck, Germany}

\thanks{$^\text{b}$Institute of Applied Mathematics and Mechanics of NAS of Ukraine,
General Batyuk Str. 19, Slov’yans’k, Donetsk region, Ukraine, 84100}


\thanks{$^1$Supported by DFG project KO 5804/1-1.}

\thanks{$^*$Corresponding author:  kolomoitsev@math.uni-luebeck.de}


\date{\today}

\begin{abstract}
In this paper asymptotic formulas are given for the Lebesgue constants generated by three special approximation processes related to the $\ell_1$-partial sums of Fourier series. In particular, we consider the Lagrange interpolation polynomials based on the  Lissajous-Chebyshev node points, the partial sums of the Fourier series generated by the anisotropically dilated rhombus, and the corresponding discrete partial sums.
\end{abstract}

\maketitle

\section{Introduction}

Let $D\subset \R^d$ be a compact set with non-empty interior. Denote by $C(D)$ the set of all real-valued continuous functions on $D$ equipped with the norm $\|f\|=\sup_{\x\in D}|f(\x)|$.
By $\Pi_{\n}$, $\n=(n_1,\dots,n_d)\in \N^d$, we denote some abstract space of polynomial functions in $C(D)$ (e.g., algebraic or trigonometric polynomials of degree $\n$ in some sense).

Consider a projection operator $P_{\n}\,:\, C(D)\mapsto \Pi_{\n}$. The norm of this operator $\Lambda_{\n}=\|P_{\n}\|=\sup_{\|f\|\le 1}\|P_{\n}(f)\|$ is called the Lebesgue constant of $P_{\n}$. In view of the well-known inequality
$$
\|f-P_{\n}(f)\|\le (1+\Lambda_{\n})\inf_{P\in \Pi_{\n}}\|f-P\|,
$$
the Lebesgue constant is an essential tool for the investigation of approximation properties of the operator $P_{\n}$.
We are interested in studying the Lebesgue constants in the case $P_{\n}$ is an interpolation process or a partial sum of the Fourier series.

For univariate interpolation processes, the behaviour of the Lebesgue constants is well studied for the most classical sets of nodes.
For example, if $P_n$ is the Lagrange interpolation polynomial with the Chebyshev nodes on $[-1,1]$, then
$$
\Lambda_n=\frac2{\pi}\log n+\frac{2}{\pi}\(\gamma+\log\frac8{\pi}\)+\mathcal{O}\(\frac1{n^2}\),
$$
where $\gamma$ denotes Euler’s constant, see, e.g.,~\cite[p.~65]{MM}. See also \cite{Sm} for related asymptotics of the Lebesgue constant of interpolation processes based on other sets of nodes.

In multivariate spaces, much less is known except the trivial case of the Lagrange interpolation based on the tensor product grid.
One of the main problems is the choice of a suitable set of nodes.
At the present time, promising Lagrange interpolation polynomials on $[-1,1]^2$ are constructed by means of the so-called Padua points ${\rm Pad}_n$,  see, e.g.,~\cite{CMV}.
Note that these polynomials have a series of favourite properties, one of which is that the Padua points can be characterized as a set of node points of a particular Lissajous curve (see~\cite{BDeMVXu2006}).
The Lebesgue constant of the Lagrange interpolation at the Padua points was studied in~\cite{BDeMVXu2006}.
In particular, it was shown that it grows like $\mathcal{O}(\log^2 n)$, where $n$ is total degree of the corresponding polynomial space $\Pi_n$.
Similar result for the Lebesgue constat of the Lagrange interpolation based on the Xu points (see~\cite{Xu1996}) was early obtained  in~\cite{BDeMV2006}. The estimates from below for different Lebesgue constants were investigated in~\cite{DEKL, Su, SV}. The corresponding results  imply that for the mentioned interpolation processes $\Lambda_n\ge c\log^2 n$.

In this paper, we study the Lebesgue constant for the interpolation processes based on the so-called
Lissajous-Chebyshev nodes ${\rm LC}_{\n}$, which represent an anisotropic analogue of the Padua points in the sense that ${\rm LC}_{n,n+1}={\rm Pad}_{n}$, see~\cite{E}. Different properties of the polynomial interpolation on the Lissajous-Chebyshev nodes have been recently obtained in~\cite{DE, DEKL, E, ErbKaethnerAhlborgBuzug2015, KLP}. Note that the study of interpolation on such nodes is motivated by applications in a novel medical imaging technology called Magnetic Particle Imaging (see, e.g.,~\cite{KB}).  Sharp estimates of the Lebesgue constant of the interpolation processes on ${\rm LC}_{\n}$ were recently established in~\cite{DEKL}. In particular, it was shown that in the multivariate case
$$
c_1(d)\prod_{j=1}^d \log n_j\le \Lambda_{\n}^{\rm LC} \le c_2(d)\prod_{j=1}^d \log n_j.
$$

The main goal of this paper is to find an asymptotic formula for $\Lambda_{\n}^{\rm LC}$ in the two dimensional case.
Our approach is similar to one given in~\cite{DEKL}. First, we investigate the Lebesgue constants for the $\ell_1$-partial sums of Fourier series and then establish relationships between the corresponding Lebesgue constants. For our purposes, we essentially extend a series of equalities and estimates obtained in~\cite{Ku} and~\cite{Ku2} for the Dirichlet kernels with the frequencies in the rhombus $\{(k_1,k_2):{|k_1|}/{n_1}+{|k_2|}/{n_2}\le 1\,\}$ and apply new methods developed recently in the papers~\cite{KL} and~\cite{KLi}.

It is worth noting that the study of the Lebesgue constants for the partial sums of Fourier series has its own great interest. There are many works dedicated to this topic (see, e.g., surveys~\cite{L} and~\cite{GaLi}). Our investigation concerns the Lebesgue constants for polyhedral partial sums of Fourier series. In particular, we are interested in the so-called triangular partial sums, which is also called the $\ell_1$-partial sums. The estimates for the Lebesgue constants of such partial sums were obtained, e.g., in~\cite{Be, DEKL, KLi, KL, Ku, NP, SV}. Different asymptotic formulas for the Lebesge constant in the case of isotropic dilations of a polyhedra were given in~\cite{KS,NP, P2, S}; for the anisotropic case see~\cite{Ku} and~\cite{Ku2} as well as the recent work~\cite{KLi}, in which an asymptotic formula is given for the Lebesgue constants generated by the anisotropically dilated $d$-dimensional simplex.


%
%
%


Throughout the paper, we use the following notation: $\T^d=[-\pi,\pi)^d$, $\T=\T^1$, and
$$
\|f\|_{L(\Omega)}=\int_{\Omega}|f(\x)|d\x.
$$
For any $m,n\in \N$, we denote $x_\mu=x_{\mu}^{(m)}={\pi \mu}/{m}$, $y_\nu=y_{\nu}^{(n)}={\pi \nu}/{n}$, and
$$
\|(a_{\mu,\nu})_{\mu,\nu}\|_{\ell^{mn}}=\frac1{mn}\sum_{\mu=0}^{m-1}\sum_{\nu=0}^{n-1}|a_{\mu,\nu}|.
$$
The floor and the fractional part functions are defined, as usual, by
$[x] =\max(m\in \Z\,:\, m\le x)$  \ {and}\  $\{x\}=x-[x],$ correspondingly.
Also, we use the notation
$\, F \lesssim G,$
with $F,G\ge 0$, for the
estimate
$\, F \le c\, G,$ where $\, c$ is an absolute positive constant.


\section{Main results}
\subsection{Polynomial interpolation on Lissajous-Chebyschev nodes}\label{secLC}


Let $\Np^2=\{(m,n)\in \N^2\,:\,  m\,\, \text{and}\,\, n\,\, \text{are relatively prime}\}$.
In what follows, for simplicity, we use the notation
\begin{equation*}
  u_k=\cos\(\frac{k\pi}{m}\),\quad v_l=\cos\(\frac{l\pi}{n}\), \quad k=0,\dots,m, \quad l=0,\dots,n,
\end{equation*}
to abbreviate the Chebyshev-Gauss-Lobatto points.
For each $(m,n)\in \Np^2$, the Lissajous-Chebyshev node points of the (degenerate) Lissajous curve $\g_{mn}(t)=\(\cos(nt), \cos(mt)\)$ are defined by
$$
{\rm LC}_{mn}=\left\{\g_{mn}\(\frac{\pi k}{mn}\),\,k=0,\dots,mn\right\}.
$$
%
%
%
\begin{figure*}[ht]
\includegraphics[scale=0.3, clip]{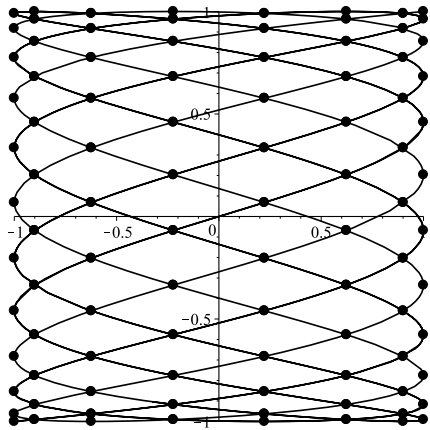}
	\caption{Illustration of the degenerate Lissajous curve $\gamma_{7,23}$ and the set ${\rm LC}_{7,23}$.}
\end{figure*}

\noindent Note that (see~\cite{E}) for every $({m,n})\in \Np^2$,
the set ${\rm LC}_{mn}$ contains  $(m+1)(n+1)/2$ distinct points and can be represented in the following form
  $$
  {\rm LC}_{mn}=\left\{(u_i,v_j): (i,j)\in \I_{mn} \right\},
  $$
where
   $$
  \I_{mn}=\left\{(i,j)\in \Z_+^2: i=0,\dots,m,\, j=0,\dots,n,\, i+j=0\pmod 2\right\}.
  $$


Let $C_n$ be the normalized Chebyshev polynomial defined by
$$
{C}_n(u)=\left\{
               \begin{array}{ll}
                 1, & \hbox{$n=0$,} \\
                 \sqrt{2}\cos(n \arccos u), & \hbox{$n\neq 0$.}
               \end{array}
             \right.
$$
A proper set of polynomials for interpolation on ${\rm LC}_{mn}$ is given by
\begin{equation*}
  \Pi_{mn}={\rm span}\{C_{k}(u)C_{l}(v)\,:\, (k,l)\in \Gamma_{mn}\},
\end{equation*}
where
$$
\G_{mn}=\left\{(i,j)\in \Z_+^2 : \frac{i}{m}+\frac{j}{n}<1\right\} \cup \{(0,n)\}.
$$



It was proved in~\cite{E} that
for every continuous function $f\,:\,[-1,1]^2\to \R$, the unique solution of the interpolation problem
\begin{equation*}
  \mathcal{P}_{mn}(f)(u_k,v_l)=f(u_k,v_l),\quad  (u_k,v_l)\in {\rm LC}_{mn},
\end{equation*}
in the space $\Pi_{nm}$ is given by the polynomial
\begin{equation*}
  \mathcal{P}_{mn}(f)(u,v)=\sum_{(k,l)\in \I_{mn}}f(u_k,v_l)\vp_{mn}(u,v;u_k,v_l),
\end{equation*}
where
  \begin{equation*}
  \begin{split}
    \vp_{mn}(u,v;u_k,v_l)&=\l_{kl}\bigg(\sum_{(i,j)\in \Gamma_{mn}} {C}_{i}(u_k){C}_{j}(v_l){C}_{i}(u){C}_{j}(v)-\frac12 {C}_{n}(v_l){C}_{n}(v)\bigg)
  \end{split}
  \end{equation*}
  and
  \begin{equation*}
  \l_{kl}:=\left\{
        \begin{array}{ll}
          \displaystyle1/{(2mn)}, & \hbox{$(u_k,v_l)$ is a vertex point of $[-1,1]^2$,} \\
          \displaystyle1/(mn), & \hbox{$(u_k,v_l)$ is an edge point of $[-1,1]^2$,} \\
          \displaystyle2/(mn), & \hbox{$(u_k,v_l)$ is an interior point of $[-1,1]^2$.}
        \end{array}
      \right.
\end{equation*}

Consider the Lebesgue constant of the above interpolation problem. Note that it can also be defined by
$$
\Lambda_{mn}^{\rm LC}:=\max_{(u,v)\in [-1,1]^2}\sum_{(k,l)\in \I_{mn}}|\vp_{mn}(u,v;u_k,v_l)|.
$$



The following theorem is our main result.

\begin{theorem}\label{th2}
  Let $m,n\in \N$ be such that $n=\l m+p$, where $\l,p\in \N$, $1\le p<m$. Then
  \begin{equation*}
    \Lambda_{mn}^{\rm LC}= \frac{4}{\pi^2}(2\log m\log n-\log^2 m)+\mathcal{O}\(\log n+p \log \frac{m}{p}\).
  \end{equation*}
\end{theorem}

In particular, Theorem~\ref{th2} implies that for the interpolation based on the Padua points ${\rm Pad}_n$, we have
$$
\Lambda_{n}^{\rm Pad}= \frac{4}{\pi^2}\log^2 n+\mathcal{O}(\log n).
$$
Similar formula holds for $\Lambda_{mn}^{\rm LC}$ if $n\sim m$ (i.e., $\l$ and $p$ are certain fixed numbers). At the same time, if $\frac {\log n}{\log m}\to \infty$ as $m\to \infty$, then
$$
\Lambda_{mn}^{\rm LC}= \frac{8}{\pi^2}\log n\log m+\mathcal{O}(\log n+\log^2 m).
$$
This formula follows from the proof of Theorem~\ref{th2} given in the last section and inequality~\eqref{le2.2}.

\subsection{$\ell_1$-partial sums of Fourier series}

Let $f$ be an integrable  $2\pi$-periodic in each variable function. The Fourier series of $f$ is given by
$$
f(x,y)\sim\sum_{k,l\in\Z} c_{kl}(f)e^{i(kx+ly)},
$$
where
$$
c_{kl}(f)=\frac1{4\pi^2}\int_{\T^2} f(x,y)e^{-i(kx+ly)}dxdy.
$$
In the multivariate case, there are many ways to define the partial sums of Fourier series (see, e.g.,~\cite{L}). In this paper, we consider the so called $\ell_1$-partial sums:
\begin{equation}\label{ps}
  \mathcal{S}_{mn}(f)(x,y)=\sum_{{|k|}/{m}+{|l|}/{n}\le 1}c_{kl}(f)e^{i(kx+ly)}.
\end{equation}
The Lebesgue constant of this partial sums is denoted by
$$
\mathcal{L}_{mn}:=\|\mathcal{S}_{mn}\|_{C(\T^2)\to C(\T^2)}=\sup_{\|f\|\le 1}|\mathcal{S}_{mn}(f)(x,y)|.
$$
It is well known that
\begin{equation}\label{zv}
  \mathcal{L}_{mn}=\frac1{4\pi^2} \int_{\T^2}|D_{mn}(x,y)|dxdy,
\end{equation}
where
$$
D_{mn}(x,y)=\sum_{{|k|}/{m}+{|l|}/{n}\le 1} e^{i(kx+ly)}
$$
is the corresponding Dirichlet kernel.

The following asymptotic equality was established in~\cite{Ku}:
\begin{equation}\label{Ku1}
  \mathcal{L}_{mn}=\frac{16}{\pi^4}(2\log m\log n-\log^2 m)+\mathcal{O}(\log n),
\end{equation}
where $m,n\in \N$ are such that $\frac n m\in \N$.

We improve this formula by considering arbitrary $m,n\in \N$.

\begin{theorem}\label{th1}
  Let $m,n\in \N$, $3\le m\le n$. Then
  \begin{equation}\label{th1.1}
  \mathcal{L}_{mn}=\frac{16}{\pi^4}(2\log m\log n-\log^2 m)+\|F_{mn}\|_{L(\T^2)}+\mathcal{O}(\log\log n \log m+\log n),
\end{equation}
where
\begin{equation}\label{F}
  \begin{split}
F_{mn}(x,y):=4\sum_{k=1}^m \left\{-\frac nmk\right\}\cos kx\cos{n\(1+\frac{k}m\)y}.
  \end{split}
\end{equation}
Moreover, if $m$ and $n$ are such that $n=\l m+p$, where $\l,p\in\N$, $1\le p\le m$. Then
  \begin{equation}\label{th1.1c}
  \mathcal{L}_{mn}=\frac{16}{\pi^4}(2\log m\log n-\log^2 m)+\mathcal{O}\(\log n+p \log \frac{m}{p}\).
\end{equation}
\end{theorem}

\begin{remark}\label{rem1}
It is known that $\|F_{mn}\|_{L(\T^2)}=\mathcal{O}(\log^2 m)$, see Lemma~\ref{le2} below. In some special cases, one can even show that this estimate is sharp, see, e.g.,~\cite{NP}. At the same time, it follows from Lemma~\ref{le3} that in the case $n=\l m+p$, where $\l,p\in\N$, one has the estimate $\|F_{mn}\|_{L(\T^2)}=\mathcal{O}(p\log \frac mp)$, which is better than $\log^2 m$ for an appropriate fixed $p$.  Generally, asymptotic properties of $\|F_{mn}\|_{L(\T^2)}$ are unknown, see, e.g.,~\cite{NP} for a discussion.
\end{remark}

\subsection{$\ell_1$-discrete partial sums of Fourier series}
The discrete analogue of partial sums~\eqref{ps} is given by
$$
\widetilde{\mathcal{S}}_{mn}(f)(x,y)=\sum_{{|k|}/{m}+{|l|}/{n}\le 1}\widetilde{c}_{kl}(f)e^{i(kx+ly)},
$$
where
$$
\widetilde{c}_{kl}(f)=\frac1{4mn}\sum_{{}_{-n\le \nu\le n-1}^{-m\le \mu\le m-1}}f(x_\mu,y_\nu)e^{-i(kx_\mu+ly_\nu)}.
$$
The Lebesgue constant of $\widetilde{\mathcal{S}}_{mn}$ can be defined by
$$
\Lambda_{mn}:=\sup_{(x,y)\in [0,\pi]^2} L_{mn}(x,y),
$$
where
$$
L_{mn}(x,y):=\sup_{\|f\|\le 1}|\widetilde{\mathcal{S}}_{mn}(f)(x,y)|=\frac1{4mn}\sum_{{}_{-n\le \nu\le n-1}^{-m\le \mu\le m-1}}|D_{mn}(x-x_\mu,y-y_\nu)|
$$
is the corresponding Lebesgue function.

The following  asymptotic equality was obtained in~\cite{Ku2}:
\begin{equation*}
\begin{split}
     L_{mn}(x,y)=\frac{2}{\pi^2}\Phi_{mn}(x,y)\(2\log m\log n-\log^2 m\)+\mathcal{O}(\log n),
\end{split}
\end{equation*}
where $m,n\in \N$ are such that $\frac n m\in \N$ and
\begin{equation}
\label{phi}
\begin{split}
\Phi_{mn}(x,y):=\Big|\sin&\frac{mx+ny}{2}\sin\frac{mx-ny}{2}\Big|+\Big|\cos\frac{mx+ny}{2}\cos\frac{mx-ny}{2}\Big|.
\end{split}
\end{equation}
In particular,
\begin{equation}\label{a2}
  \Lambda_{mn}=\frac{2}{\pi^2}\(2\log m\log n-\log^2 m\)+\mathcal{O}(\log n).
\end{equation}

In the next theorem, we obtain an analogue of~\eqref{a2} for any $m$ and $n$.

\begin{theorem}\label{th1+}
  Let $m,n\in \N$, $3\le m\le n$. Then
\begin{equation}\label{th1.2}
\begin{split}
     \Lambda_{mn}=\frac{2}{\pi^2}&\(2\log m\log n-\log^2 m\)+\mathfrak{F}_{mn}+\mathcal{O}(\log\log n \log m+\log n),
\end{split}
\end{equation}
where
$$
\mathfrak{F}_{mn}:=\sup_{(x,y)\in \T^2}\frac1{mn}\sum_{{}_{0\le \nu\le n-1}^{0\le \mu\le m-1}}|F_{mn}(x+x_\mu,y+y_\nu)|
$$
and $F_{mn}$ is defined in~\eqref{F}.
Moreover, if $m$ and $n$ are such that $n=\l m+p$, where $\l,p\in\N$, $1\le p\le m$. Then
\begin{equation}\label{th1.2c}
\begin{split}
     \Lambda_{mn}=\frac{2}{\pi^2}&\(2\log m\log n-\log^2 m\)+\mathcal{O}\(\log n+p \log \frac{m}{p}\).
\end{split}
\end{equation}
\end{theorem}

\begin{remark}\label{rem2}
The same comments as in Remark~\ref{rem1} hold true for the quantity $\mathfrak{F}_{mn}$.
To see this, it suffices to apply the Marcinkiewicz-Zygmund  inequality (see Lemma~\ref{lemz}), which shows that $\mathfrak{F}_{mn}\lesssim \sup_{y\in \T}\|F_{mn}(\cdot,y)\|_{L(\T)}$, and then to use Lemma~\ref{le3}.
\end{remark}

\section{Auxiliary results}


For any $m,n\in \N$, we denote
\begin{equation}\label{S}
  \begin{split}
S_{mn}(x,y):=\frac2yD_m\(x+\frac nmy\)D_m\(x-\frac nmy\)\sin\frac nmy,
  \end{split}
\end{equation}
where
$$
D_m(x):=
\frac{\sin\frac{m x}2}{\sin \frac x2},
$$
and
\begin{equation}\label{R}
  \begin{split}
R_{mn}(x,y):=\sum_{\nu\neq 0}\frac{y}{\pi \nu (2\pi\nu+y)}\sum_{k=-m}^m e^{ikx}\sin&\(n\(1-\frac{|k|}m\)(2\pi\nu+y)\)\\
&+\sum_{k=-m}^m e^{ikx}\cos\(n\(1-\frac{|k|}m\)y\).
  \end{split}
\end{equation}

\begin{lemma}\label{le1}
For any $m,n\in \N$,  we have
\begin{equation}\label{le1.1}
  D_{mn}(x,y)=S_{mn}(x,y)-F_{mn}(x,y)+R_{mn}(x,y),
\end{equation}
where $F_{mn}$ is defined in~\eqref{F}.
\end{lemma}

\begin{proof}
We obtain
\begin{equation*}
  \begin{split}
     D_{mn}(x,y)&=\sum_{k=-m}^m e^{ikx}\sum_{l=-[n(1-\frac{|k|}m)]}^{[n(1-\frac{|k|}m)]}e^{ily}=\sum_{k=-m}^m e^{ikx}\int_{-n(1-\frac{|k|}m)}^{n(1-\frac{|k|}m)}e^{ily}d[l]\\
     &=\sum_{k=-m}^m e^{ikx}\int_{-n(1-\frac{|k|}m)}^{n(1-\frac{|k|}m)}e^{ily}dl-\sum_{k=-m}^m e^{ikx}\int_{-n(1-\frac{|k|}m)}^{n(1-\frac{|k|}m)}e^{ily}d\{l\}\\
     &:=J_1(x,y)-J_2(x,y).
  \end{split}
\end{equation*}
Consider $J_1$. We have
\begin{equation}\label{2}
  \begin{split}
     J_{1}(x,y)&=\frac2y\sum_{k=-m}^m e^{ikx}\sin n\(1-\frac{|k|}m\)y\\
     &=\frac{2\sin ny}{y}+\frac4y\sum_{k=1}^m \cos kx\sin n\(1-\frac{k}m\)y.
  \end{split}
\end{equation}
Next, using~\cite[AD(361.8)]{GR} and the standard trigonometric identities, we derive
\begin{equation}\label{3}
  \begin{split}
     2&\sum_{k=1}^m \cos kx \sin n\(1-\frac{k}m\)y\\
     &=\sum_{k=1}^m \(\sin\(\big(x-\frac nm y\big)k+ny\)-\sin\(\big(x+\frac nm y\big)k-ny\)\)\\
     &=\sin\(\frac{m+1}{2}\big(x-\frac nmy\big)+ny\)D_m\(x-\frac nmy\)\\
     &\qquad\qquad\qquad\qquad\qquad-\sin\(\frac{m+1}{2}\big(x+\frac nmy\big)-ny\)D_m\(x+\frac nmy\)\\
     &=\sin\(\frac{mx+ny}{2}+\frac{1}{2}\big(x-\frac nmy\big)\)D_m\(x-\frac nmy\)\\
     &\qquad\qquad\qquad\qquad\qquad-\sin\(\frac{mx-ny}{2}+\frac{1}{2}\big(x+\frac nmy\big)\)D_m\(x+\frac nmy\)\\
     &=\sin\(\frac{mx+ny}{2}\)\sin\(\frac{mx-ny}{2}\)\(\cot\frac{1}{2}\big(x-\frac nmy\big)-\cot\frac{1}{2}\big(x+\frac nmy\big)\)\\
     &\qquad\qquad+\sin\(\frac{mx-ny}{2}\)\cos\(\frac{mx+ny}{2}\)-\cos\(\frac{mx-ny}{2}\)\sin\(\frac{mx+ny}{2}\)\\
     &=D_m\(x+\frac nmy\)D_m\(x-\frac nmy\)\sin\frac nmy-\sin ny.
  \end{split}
\end{equation}
Combining~\eqref{2} and~\eqref{3}, we get
\begin{equation}\label{4}
  \begin{split}
J_1(x,y)=\frac2yD_m\(x+\frac nmy\)D_m\(x-\frac nmy\)\sin\frac nmy=S_{mn}(x,y).
  \end{split}
\end{equation}
Consider $J_2$. Integrating by parts, we obtain
\begin{equation}\label{5}
  \begin{split}
\int_{-n(1-\frac{|k|}m)}^{n(1-\frac{|k|}m)}e^{ily}d\{l\}=2\left\{-\frac nm|k|\right\}\cos{n\(1-\frac{|k|}m\)y}-&e^{-in(1-\frac{|k|}m)y}\\
&-iy\int_{-n(1-\frac{|k|}m)}^{n(1-\frac{|k|}m)}\{l\}e^{ily}dl.
  \end{split}
\end{equation}
Next, representing $\{\cdot\}$ via the Fourier series as
$$
\{\xi\}=\frac12+\sum_{\nu\neq 0}\frac1{2\pi i \nu}e^{2\pi i\nu \xi},
$$
we derive
\begin{equation}\label{6}
  \begin{split}
\int_{-n(1-\frac{|k|}m)}^{n(1-\frac{|k|}m)}\{l\}e^{ily}dl&=\frac1y\sin n\(1-\frac{|k|}m\)y+\sum_{\nu\neq 0}\frac1{2\pi i \nu}\int_{-n(1-\frac{|k|}m)}^{n(1-\frac{|k|}m)}e^{2\pi i\nu l}e^{ily}dl\\
&=\frac1y\sin n\(1-\frac{|k|}m\)y+\sum_{\nu\neq 0}\frac{\sin\((2\pi\nu+y)n\(1-\frac{|k|}m\)\)}{\pi i \nu (2\pi\nu+y)}.
  \end{split}
\end{equation}
Combining~\eqref{5} and~\eqref{6}, we get
\begin{equation*}
  \begin{split}
\int_{-n(1-\frac{|k|}m)}^{n(1-\frac{|k|}m)}e^{ily}d\{l\}=
\(2\left\{-\frac nm|k|\right\}-1\)&\cos{n\(1-\frac{|k|}m\)y}\\
&-y\sum_{\nu\neq 0}\frac{\sin\((2\pi\nu+y)n\(1-\frac{|k|}m\)\)}{\pi \nu (2\pi\nu+y)}.
  \end{split}
\end{equation*}
This implies that
\begin{equation}\label{8}
  \begin{split}
&J_2(x,y)=\sum_{k=-m}^m e^{ikx}\(2\left\{-\frac nm|k|\right\}-1\)\cos{n\(1-\frac{|k|}m\)y}\\
&-\sum_{\nu\neq 0}\frac{y}{\pi \nu (2\pi\nu+y)}\sum_{k=-m}^m e^{ikx}\sin\((2\pi\nu+y)n\(1-\frac{|k|}m\)\)\\
&=F_{mn}(x,y)-\sum_{\nu\neq 0}\frac{y}{\pi \nu (2\pi\nu+y)}\sum_{k=-m}^m e^{ikx}\sin\((2\pi\nu+y)n\(1-\frac{|k|}m\)\)\\
&\qquad\qquad\quad\,\,\,-\sum_{k=-m}^m e^{ikx}\cos{n\(1-\frac{|k|}m\)y}.
  \end{split}
\end{equation}
Finally, combining~\eqref{4} and~\eqref{8}, we prove the lemma.
\end{proof}

\begin{remark}
It follows from~\eqref{3} and~\eqref{4} that 
\begin{equation}\label{S2}
  \begin{split}
S_{mn}(x,y)&=\frac2y\Bigg(\sin\(\frac{m+1}{2}\big(x-\frac nmy\big)+ny\)D_m\(x-\frac nmy\)\\
     &\qquad\qquad\qquad-\sin\(\frac{m+1}{2}\big(x+\frac nmy\big)-ny\)D_m\(x+\frac nmy\)+\sin ny\Bigg).
  \end{split}
\end{equation}
\end{remark}

\medskip

Denote by $\mathcal{D}_m$ the following modification of the Dirichlet kernel ${D}_m$:
$$
\mathcal{D}_m(x):=\sum_{k=1}^m e^{ikx}.
$$

\begin{lemma}\label{le2}
  Let $m,n\in \N$ be such that $3\le m\le n$. Then
    \begin{equation}\label{le2.1}
    \|\mathcal{D}_{m}\|_{L(\T)}\lesssim \log m,
  \end{equation}
  \begin{equation}\label{le2.2}
    \|F_{mn}\|_{L(\T^2)}\lesssim \log^2 m,
  \end{equation}
    \begin{equation}\label{le2.3}
    \|R_{mn}\|_{L(\T^2)}\lesssim \log m.
  \end{equation}
\end{lemma}

\begin{proof}
Inequality~\eqref{le2.1} is well known, see, e.g.,~\cite[Ch.~II, \S~12]{Z}.  For the proof of~\eqref{le2.2}, see Lemma~3.9 in~\cite{KL}.
Inequality~\eqref{le2.3} follows easily from~\eqref{R} and~\eqref{le2.1}.
\end{proof}

Denote
$$
\mathcal{F}_{mp}(x):=\sum_{k=0}^m \Big\{-\frac{pk}{m}\Big\}e^{ikx}.
$$

\begin{lemma}\label{le3}
  Let $m,p\in \N$. Then
  \begin{equation}\label{le3.1}
    \| \mathcal{F}_{mp}\|_{L(\T)}\lesssim \log^2 m.
  \end{equation}
Moreover,  if $1\le p<m$. Then
  \begin{equation}\label{le3.2}
    \| \mathcal{F}_{mp}\|_{L(\T)}\lesssim p\log\frac mp.
  \end{equation}
 In the above inequalities, the constant in $\lesssim$ does not depend on $p$ and $m$.
\end{lemma}

\begin{proof}
The proof of~\eqref{le3.1} can be found in~\cite{KL} (cf.~\eqref{le2.2}). Let us prove~\eqref{le3.2}.
We have
\begin{equation}\label{9}
  \begin{split}
    \mathcal{F}_{mp}(x)&=\sum_{l=0}^{p-1}\sum_{k=[\frac{lm}p]+1}^{[\frac{(l+1)m}p]} \bigg\{-\frac{pk}{m}\bigg\}e^{ikx}
    \\
    &=\sum_{l=0}^{p-1}e^{i[\frac{lm}p]x}\sum_{k=1}^{[\frac{(l+1)m}p]-[\frac{lm}p]} \bigg\{-\frac{p}{m}\Big(k-\Big\{\frac{lm}{p}\Big\}\Big)\bigg\}e^{ikx}.
  \end{split}
\end{equation}
Since $1\le k\le \frac mp-\{\frac{(l+1)p}m\}+\{\frac{lp}m\}$, we get  $0<\frac{p}{m}(k-\{\frac{lm}{p}\})\le 1-\frac pm\{\frac{(l+1)p}m\}< 1$, which together with~\eqref{9} implies
\begin{equation*}
  \begin{split}
    \mathcal{F}_{mp}(x)&=-\sum_{l=0}^{p-1}e^{i[\frac{lm}p]x}\sum_{k=1}^{[\frac{(l+1)m}p]-[\frac{lm}p]} \frac{p}{m}\Big(k-\Big\{\frac{lm}{p}\Big\}\Big)e^{ikx}\\
    &=-\frac pm\sum_{l=0}^{p-1}e^{i[\frac{lm}p]x}\(\mathcal{D}'_{[\frac{(l+1)m}p]-[\frac{lm}p]}(x)-\Big\{\frac{lm}{p}\Big\}
    \mathcal{D}_{[\frac{(l+1)m}p]-[\frac{lm}p]}(x)\).
  \end{split}
\end{equation*}
Thus, using the Bernstein inequality (see, e.g.,~\cite[Ch.~10, \S.~3]{Z}) and~\eqref{le2.1}, we derive
  \begin{equation*}
  \begin{split}
        \|\mathcal{F}_{mp}\|_{L(\T)}\le \frac pm\sum_{l=0}^{p-1}\(\Big\|\mathcal{D}'_{[\frac{(l+1)m}p]-[\frac{lm}p]}\Big\|_{L(\T)}
        +\Big\|\mathcal{D}_{[\frac{(l+1)m}p]-[\frac{lm}p]}\Big\|_{L(\T)}\)\lesssim p\log\frac mp,
  \end{split}
  \end{equation*}
 which proves the lemma. 
\end{proof}

Recall the well-known Marcinkiewicz-Zygmund inequality for trigonometric polynomials (see, e.g.,~\cite[4.3.3]{TB}).

\begin{lemma}\label{lemz}
  Let $T$ be a trigonometric polynomial of degree at most~$n$. Then
  \begin{equation*}
    \frac1n\sum_{\nu=0}^{n-1} |T(\pi\nu/n)|\lesssim \|T\|_{L(\T)}.
  \end{equation*}
\end{lemma}


\begin{lemma}\label{le4}
  Let $m,n\in \N$ be such that $3\le m\le n$. Then
  \begin{equation}\label{le4.1}
    \|D_{mn}\|_{L([0,\pi)^2)}=\|S_{mn}\|_{L([0,\pi)^2)}+\|F_{mn}\|_{L([0,\pi)^2)}+\mathcal{O}(\log\log n \log m).
  \end{equation}
If, additionally, $|x|\le {\pi}/{(2m)}$ and $|y|\le {\pi}/{(2n)}$, then
    \begin{equation}\label{le4.2}
    \begin{split}
          \|(D_{mn}(x+x_\mu,y+y_\nu))\|_{\ell^{mn}}&=\|(S_{mn}(x+x_\mu,y+y_\nu))\|_{\ell^{mn}}\\
          &+\|(F_{mn}(x+x_\mu,y+y_\nu))\|_{\ell^{mn}}+\mathcal{O}(\log\log n \log m).
    \end{split}
  \end{equation}
Moreover, if $m$ and $n$ are such that $n=\l m+p$, $p<m$, $\l,p\in \N$. Then
  \begin{equation}\label{le4.1+}
    \|D_{mn}\|_{L([0,\pi)^2)}=\|S_{mn}\|_{L([0,\pi)^2)}+\mathcal{O}\(\log m+p\log \frac mp\)
  \end{equation}
and
    \begin{equation}\label{le4.2+}
          \|(D_{mn}(x+x_\mu,y+y_\nu))\|_{\ell^{mn}}=\|(S_{mn}(x+x_\mu,y+y_\nu))\|_{\ell^{mn}}+\mathcal{O}\(\log m+p\log \frac mp\).
  \end{equation}
\end{lemma}

\begin{proof}
Using~\eqref{le3.1}, we have
\begin{equation*}
  \begin{split}
      \int_0^\pi dx \int_0^{\frac1{\log n}}|F_{mn}(x,y)|dy\le 4\int_0^{\frac1{\log n}} dy \int_\T \bigg|\sum_{k=1}^m \Big\{-\frac{nk}{m}\Big\}e^{ikx}\bigg|dx\lesssim \log m.
   \end{split}
\end{equation*}
This along with~\eqref{le1.1} and~\eqref{le2.3} implies
\begin{equation}\label{11}
  \begin{split}
      \int_0^\pi dx \int_0^{\frac1{\log n}}|D_{mn}(x,y)|dy=\int_0^\pi dx &\int_0^{\frac1{\log n}}|S_{mn}(x,y)|dy\\
      &+\int_0^\pi dx \int_0^{\frac1{\log n}}|F_{mn}(x,y)|dy+\mathcal{O}(\log m).
  \end{split}
\end{equation}
At the same time using~\eqref{S2} and~\eqref{le2.1}, we get
\begin{equation*}
  \int_0^\pi dx \int_{\frac1{\log n}}^\pi|S_{mn}(x,y)|dy\lesssim  \log\log n \int_0^\pi |D_m(x)|dx\lesssim \log m\log\log n.
\end{equation*}
As above, this along with~\eqref{le1.1} and~\eqref{le2.3} implies
\begin{equation}\label{13}
  \begin{split}
      \int_0^\pi dx \int_{\frac1{\log n}}^\pi|D_{mn}(x,y)|dy&=\int_0^\pi dx \int_{\frac1{\log n}}^\pi|S_{mn}(x,y)|dy\\
      &+\int_0^\pi dx \int_{\frac1{\log n}}^\pi|F_{mn}(x,y)|dy+\mathcal{O}(\log m\log\log n).
  \end{split}
\end{equation}
Thus, combining~\eqref{11} and~\eqref{13}, we obtain~\eqref{le4.1}.

The proof of~\eqref{le4.2} is similar.
Using Lemma~\ref{lemz} and~\eqref{le3.1}, we derive after simple calculations that
\begin{equation}\label{14}
  \begin{split}
    \frac1{mn}\sum_{\mu=0}^{m-1}\sum_{\nu=0}^{[\frac n{\log n}]-1}|F_{mn}(x+x_\mu,y+y_\nu)|\lesssim \frac1n \sum_{\nu=0}^{[\frac n{\log n}]-1} \int_{-\pi}^\pi \bigg|\sum_{k=1}^m \Big\{-\frac{nk}{m}\Big\}e^{ikt}\bigg|dt\lesssim \log m.
  \end{split}
\end{equation}
Next, using~\eqref{S2}, \eqref{le2.1}, and Lemma~\ref{lemz}, we obtain
\begin{equation}\label{15}
  \begin{split}
     \frac1{mn}&\sum_{\nu=[\frac n{\log n}]}^{n-1}\sum_{\mu=0}^{m-1}|S_{mn}(x+x_\mu,y+y_\nu)|\\
     &\le \frac{2}{mn}\sum_{\nu=[\frac n{\log n}]}^{n-1}\frac1{|y+y_\nu|}\sum_{\mu=0}^{m-1}\bigg(\Big|D_{m}\Big(x-\frac{n}{m}y+x_{\mu-\nu}\Big)\Big|\\
     &\qquad\qquad\qquad\qquad\qquad\qquad+\Big|D_{m}\Big(x+\frac{n}{m}y+x_{\mu+\nu}\Big)\Big|+|\sin n(y+y_\nu)| \bigg)\\
     &\lesssim \sum_{\nu=[\frac n{\log n}]}^{n-1} \frac1\nu \int_\T \bigg(\Big|D_{m}\Big(x-\frac{n}{m}y+t-x_{\nu}\Big)\Big|
     +\Big|D_{m}\Big(x+\frac{n}{m}y+t+x_{\nu}\Big)\Big|+1 \bigg)dt\\
     &\lesssim \log\log n\log m.
  \end{split}
\end{equation}
Thus, combining~\eqref{14} and~\eqref{15} with~\eqref{le1.1} and using the same arguments as in  the proof of~\eqref{le4.1}, we get~\eqref{le4.2}.


The proofs of~\eqref{le4.1+} and~\eqref{le4.2+}, follow from equality~\eqref{le1.1}, inequalities~\eqref{le2.3}, \eqref{le3.2}, and the Marcinkiewicz-Zygmund inequality given in Lemma~\ref{lemz}, which shows that $\sup_{x,y}\|(F_{mn}(x+x_\mu,y+y_\nu))\|_{\ell^{mn}}\lesssim p\log \frac mp$.
\end{proof}

We need the following additional notations:
$$
\D_m^{(1)}(x,y)=\frac12S_{mm}(x,y),
$$
$$
\D_m^{(2)}(x,y)=D_{m}(x-y)\sin\(\frac{m(x+y)}{2}\),
$$
$$
A_m=\{(\mu,\nu)\in \Z_+^2\,:\, \mu,\nu=0,\dots, m-1,\,|\mu-\nu|\ge 2\},
$$
and
$$
B_m=\{(\mu,\nu)\in A_m\,:\, \mu+\nu = 0 \pmod 2\,\}.
$$

\begin{lemma}\label{leKu}
Let $m\in \N$, $m\ge 3$. Then
\begin{align}
    &\|\D_m^{(1)}\|_{L([0,\pi)^2)}=\frac8{\pi^2}\log^2 m+\mathcal{O}(\log m),\label{leKu1+}
    \\
    &\|\D_m^{(2)}\|_{L([0,\pi)^2)}=\frac{16}{\pi}\log m+\mathcal{O}(1).\label{leKu2+}
\end{align}
If, additionally,  $|x|\le \pi/(2m)$ and $|y|\le \pi/(2m)$. Then
\begin{align}
    &\frac1{m^2}\sum\limits_{A_m}|\D_m^{(1)}(x+x_\mu,y+x_\nu)|=\frac1{\pi^2}\Phi_{mm}(x,y)\log^2 m+\mathcal{O}(\log m),\label{leKu1}
    \\
    &\frac1{m^2}\sum\limits_{A_m}|\D_m^{(2)}(x+x_\mu,y+x_\nu)|=\frac2{\pi}\Phi_{mm}(x,y)\log m+\mathcal{O}(1),\label{leKu2}
    \\
    &\frac1{m^2}\sum\limits_{B_m}|\D_m^{(1)}(x+x_\mu,y+x_\nu)|=\frac{1}{\pi^2}\Big|\sin{\frac{m(x+y)}{2}}\sin\frac{m(x-y)}{2}\Big|\log^2 m+\mathcal{O}(\log m),\label{leKu3}
    \\
    &\frac1{m^2}\sum\limits_{B_m}|\D_m^{(2)}(x+x_\mu,y+x_\nu)|=\frac{2}{\pi}\Big|\sin\frac{m(x+y)}{2}\sin\frac{m(x-y)}{2}\Big|\log m+\mathcal{O}(1),\label{leKu4}
\end{align}
where $\Phi_{mm}$ is defined in~\eqref{phi}.
\end{lemma}

\begin{proof}
Equalities~\eqref{leKu1+} and~\eqref{leKu2+} can be found in~\cite{Ku}. The proofs of relations~\eqref{leKu1}--\eqref{leKu4} are given in~\cite{Ku2}.
\end{proof}

\begin{lemma}\label{le5}
Let $m,n\in \N$ be such that $3\le m\le n$. Then
\begin{equation}\label{le5.1}
\|S_{mn}\|_{L([0,\pi)^2)}=\frac{16}{\pi^2}(2\log m\log n-\log^2 m)+\mathcal{O}(\log n).
\end{equation}
If, additionally, $|x|\le {\pi}/{(2m)}$ and $|y|\le {\pi}/{(2n)}$, then
    \begin{equation}\label{le5.2}
    \begin{split}
          \|(S_{mn}(x+x_\mu,&y+y_\nu))\|_{\ell^{mn}}=\frac{2}{\pi^2}\Phi_{mn}(x,y)(2\log m\log n-\log^2 m)+\mathcal{O}(\log n).
    \end{split}
  \end{equation}
\end{lemma}

\begin{proof}
 We have
  \begin{equation}\label{16}
    \begin{split}
      \|S_{mn}\|_{L([0,\pi)^2)}&=\int_0^\pi dx\int_0^{\frac nm\pi}|S_{mm}(x,y)|dy\\
      &=\int_0^\pi \int_0^{\pi}|\dots|+\int_0^\pi \int_{\pi}^{[\frac nm]\pi}|\dots|+\int_0^\pi \int_{[\frac nm]\pi}^{\frac nm\pi}|\dots|=I_0+I_1+I_2.
    \end{split}
  \end{equation}
Using~\eqref{leKu1+}, we get
  \begin{equation}\label{I0}
    \begin{split}
I_0=2\|\D_m^{(1)}\|_{L([0,\pi)^2)}=\frac{16}{\pi^2}\log^2 m+\mathcal{O}(\log m).
    \end{split}
  \end{equation}
Consider $I_1$. Denoting $l=[\frac nm]$, we derive
\begin{equation}\label{I1}
  \begin{split}
    I_1&=\sum_{k=1}^{l-1}\int_{0}^\pi dx\int_{k\pi}^{^{(k+1)\pi}}|S_{mm}(x,y)|dy\\
    &=2\sum_{k=1}^{l-1}\int_0^\pi\int_0^\pi|D_m(x+y+\pi k)D_m(x-y-\pi k)|\frac{\sin y}{y+\pi k}dxdy\\
    &=\frac{2}{\pi}\sum_{k=1}^{l-1}\frac1k\int_0^\pi\int_0^\pi|D_m(x+y+\pi k)D_m(x-y-\pi k)|\sin y\,dxdy+Q_m,\\
  \end{split}
\end{equation}
where
$$
Q_m=2\sum_{k=1}^{l-1}\int_0^\pi\int_0^\pi|D_m(x+y+\pi k)D_m(x-y-\pi k)|\sin y\(\frac{1}{y+\pi k}-\frac1{k\pi}\)dxdy.
$$
Observe that
\begin{equation}\label{o1}
\begin{split}
     \int_0^\pi\int_0^\pi|D_m(x+y+\pi k)&D_m(x-y-\pi k)|\sin y\,dxdy\\
     &=\int_0^\pi\int_0^\pi|D_m(x+y)D_m(x-y)|\sin y\, dxdy.
\end{split}
\end{equation}
This together with the equality
\begin{equation}\label{o2}
\begin{split}
\sin y=\sin\frac{x+y}{2}\cos\frac{x-y}{2}-\sin\frac{x-y}{2}\cos\frac{x+y}{2}
\end{split}
\end{equation}
and estimate~\eqref{le2.1} implies the estimate $|Q_m|\lesssim \log m$. Thus, using again~\eqref{o1}, we derive from~\eqref{I1}:
\begin{equation}\label{I1+}
  \begin{split}
     I_1=\frac2\pi\log \frac nm \int_0^\pi\int_0^\pi|D_m(x+y)D_m(x-y)|\sin y\,dxdy+\mathcal{O}(\log m).
  \end{split}
\end{equation}
Next, applying~\eqref{o2} and the fact that
$$
\int_0^\pi\int_0^\pi\bigg|D_m(x+y)\sin\frac m2(x-y)\bigg|\cos\frac{x+y}{2}\,dxdy=0,
$$
we obtain
\begin{equation}\label{I1++}
  \begin{split}
     \int_0^\pi\int_0^\pi&|D_m(x+y)D_m(x-y)|\sin y\,dxdy\\
     &=\int_0^\pi\int_0^\pi\bigg|D_m(x-y)\sin\frac m2(x+y)\bigg|\cos\frac{x-y}{2}\,dxdy\\
     &=\|\D_m^{(2)}\|_{L([0,\pi)^2)}-2\int_0^\pi\int_0^\pi|\D_m^{(2)}(x,y)|\sin^2\frac{x-y}{4}\,dxdy\\
     &=\|\D_m^{(2)}\|_{L([0,\pi)^2)}-\int_0^\pi\int_0^\pi\bigg|\sin\frac m2(x+y)\sin\frac m2(x-y)\bigg|\tan\frac{x-y}{4}\,dxdy\\
     &=\|\D_m^{(2)}\|_{L([0,\pi)^2)}+\mathcal{O}(1).
  \end{split}
\end{equation}
Now, combining~\eqref{I1+}, \eqref{I1++} and~\eqref{leKu2+}, we get
\begin{equation}\label{I1+++}
  \begin{split}
    I_1=\frac{32}{\pi^2}\log n\log m+\mathcal{O}(\log n).
  \end{split}
\end{equation}
Consider $I_2$. Using~\eqref{S2} and estimate~\eqref{le2.1}, we derive
  \begin{equation}\label{18}
    \begin{split}
      I_2\lesssim \int_0^\pi dx \int_{[\frac nm]\pi}^{\frac nm\pi}\(|D_m(x+y)|+|D_m(x-y)|+1\)dy\lesssim \Vert D_m\Vert_{L(\T)}\lesssim \log m.
    \end{split}
  \end{equation}
  Thus, combining \eqref{16}, \eqref{I0}, \eqref{I1+++}, and~\eqref{18}, we obtain~\eqref{le5.1}.

\smallskip

Now we prove~\eqref{le5.2}. As above, we have
    \begin{equation}\label{19}
    \begin{split}
          \|(S_{mn}(x+x_\mu,y+y_\nu))\|_{\ell^{mn}}=\frac1{mn}&\sum_{\mu=0}^{m-1}\sum_{\nu=0}^{m-1}\!|\dots|+\frac1{mn}\sum_{\mu=0}^{m-1}\sum_{\nu=m}^{m[\frac nm]-1}\!|\dots|\\
          &+\frac1{mn}\sum_{\mu=0}^{m-1}\sum_{\nu=m[\frac nm]}^{n-1}\!|\dots|=J_0+J_1+J_2.
    \end{split}
  \end{equation}
Here, we suppose that $\sum_{\nu=A}^B=0$ if $A>B$.  Repeating arguments similar to those made in the proof of~\eqref{le5.1} (see equalities for $I_0$ and $I_1$) and using~\eqref{leKu1} and~\eqref{leKu2}, 
we obtain
    \begin{equation*}
    \begin{split}
  J_0+J_1&=\frac2{m^2}\sum_{A_m}\bigg|\D_m^{(1)}\Big(x+x_\mu,\frac nm\Big(y+y_\nu\Big)\Big)\bigg|\\
  &\qquad\qquad+\(\log \frac nm\)\frac{2}{\pi m^2} \sum_{A_m}\bigg|\D_m^{(2)}\Big(x+x_\mu,\frac nm\Big(y+y_\nu\Big)\Big)\bigg|+\mathcal{O}(\log m)\\
  &=\frac{2}{\pi^2}\Phi_{mn}(x,y)(2\log m\log n-\log^2 m)+\mathcal{O}(\log n).
    \end{split}
  \end{equation*}
To estimate $J_2$, we apply~\eqref{S2}, Lemma~\ref{lemz}, and~\eqref{le2.1}:
    \begin{equation}\label{22}
    \begin{split}
J_2&\lesssim \frac 1m\sum_{\nu=0}^{m-1}\sum_{\nu=m[\frac nm]}^{n-1} \frac1\nu\ \bigg(\Big|D_m\Big(x-\frac nmy+x_\mu-x_\nu\Big)\Big|\\
&\qquad\qquad\qquad\qquad\qquad\qquad\qquad+\Big|D_m\Big(x+\frac nmy+x_\mu+x_\nu\Big)\Big|+1\bigg)\\
&\lesssim \log \(\frac nm \Big[\frac nm\Big]^{-1}+1\)\|D_m\|_{L_1(\T)}\lesssim \log m.
    \end{split}
  \end{equation}
Finally, combining~\eqref{19} and~\eqref{22}, we get~\eqref{le5.2}.
\end{proof}

\section{Proofs of the main results}
\begin{proof}[Proof of Theorem~\ref{th2}]
Using the substitution $u=\cos x$, $v=\cos y$, and denoting
$$
\mathcal{D}_{mn}(x,y)=\sum_{{|k|}/{m}+{|l|}/{n}<1} e^{i(k x+l y)},
$$
we obtain
\begin{equation}\label{23}
  \begin{split}
    \Lambda_{mn}^{\rm LC}&=\max_{(x,y)\in [0,\pi]^2}\sum_{(\mu,\nu)\in \I_{m,n}} \!\!\l_{\mu,\nu}\bigg|\sum_{(i,j)\in \Gamma_{m,n}}
    4\cos ix_\mu\cos j y_\nu\cos ix\cos jy -\cos ny_\nu\cos ny\bigg|\\
    &=\max_{(x,y)\in [0,\pi]^2}\sum_{(\mu,\nu)\in \I_{m,n}} \!\!\l_{\mu,\nu} \bigg|\sum_{(i,j)\in \Gamma_{m,n}}
    4\cos ix_\mu\cos j y_\nu\cos ix\cos jy\bigg|+\mathcal{O}(1)\\
    &=\max_{(x,y)\in [0,\pi]^2}\sum_{(\mu,\nu)\in \I_{m,n}} \!\!\frac{\l_{\mu,\nu}}{4} |\mathcal{D}_{mn}(x_\mu+x,y_\nu+y)+\mathcal{D}_{mn}(x_\mu+x,y_\nu-y)\\
    &\qquad\qquad\qquad\quad\,+\mathcal{D}_{mn}(x_\mu-x,y_\nu+y)+\mathcal{D}_{mn}(x_\mu-x,y_\nu-y)|+\mathcal{O}(1)\\
    &\le\max_{(x,y)\in [0,\pi]^2}\sum_{(\mu,\nu)\in \I_{m,n}} \!\!\l_{\mu,\nu} |\mathcal{D}_{mn}(x_\mu+x,y_\nu+y)|+\mathcal{O}(1).
  \end{split}
\end{equation}
Denote by ${\rm LC}_{mn}'$ a subset of ${\rm LC}_{mn}$, which consists of vertex and edge points of $[-1,1]^2$. It was shown in~\cite{E} that
$$
{\rm LC}_{mn}'=\left\{(u_i,v_j): \begin{array}{ccc}
                                              i=1,\dots,m &  \\
                                              j\in \{0,n\} & \\
                                              i+j=0\pmod 2 &
                                            \end{array}
  \right\}\cup
 \left\{(u_i,v_j): \begin{array}{ccc}
                                              i\in \{0,m\} &  \\
                                              j=1,\dots,n-1 & \\
                                              i+j=0\pmod 2 &
                                            \end{array}
  \right\}.
  $$
Taking into account this equality, using Lemma~\ref{lemz}, and denoting
$$
\|(a_{\nu,\mu})\|_{\widetilde{\ell}^{mn}}=\frac1{mn}\sum_{{\nu,\mu}\in \mathcal{I}_{mn}}|a_{\nu,\mu}|,
$$
we get from~\eqref{23} that
\begin{equation}\label{23+}
  \begin{split}
    \Lambda_{mn}^{\rm LC}
    &\le\max_{(x,y)\in [0,\pi]^2}2\|\(\mathcal{D}_{mn}(x_\mu+x,y_\nu+y)\)\|_{\widetilde{\ell}^{mn}}\\
    &\qquad\qquad\qquad\qquad+\mathcal{O}\(m^{-1}\|\mathcal{D}_n\|_{L(\T)}+n^{-1}\|\mathcal{D}_m\|_{L(\T)}+1\)\\
    &=\max_{(x,y)\in [0,\pi]^2}2\|\(\mathcal{D}_{mn}(x_\mu+x,y_\nu+y)\)\|_{\widetilde{\ell}^{mn}}+\mathcal{O}(\log n).
  \end{split}
\end{equation}
Next, we have
\begin{equation}\label{24}
  \begin{split}
     \mathcal{D}_{mn}(x,y)=D_{mn}(x,y)-d_{mn}(x,y),
  \end{split}
\end{equation}
where
\begin{equation*}
  d_{mn}(x,y)=\sum_{{|k|}/{m}+{|l|}/{n}=1} e^{i(k x+l y)}=2\sum_{k=-m}^m e^{ik x}\cos\(\left\{-\frac{n}{m}|k|\right\}\).   
\end{equation*}
Repeating the proof of Lemma~\ref{le3}, it is not difficult to show that
\begin{equation}\label{26}
  \sup_{y\in \T}\|d_{mn}(\cdot,y)\|_{L_1(\T)}\lesssim p \log \frac m p.
\end{equation}
Thus, using~\eqref{24} and~\eqref{26} as well as an analogue of~\eqref{le4.2+} with the norm $\widetilde{\ell}_{mn}$ instead of ${\ell}_{mn}$ (to obtain the required equality, it suffices to repeat line by line the proof of~\eqref{le4.2+}), we derive
\begin{equation}\label{27}
\begin{split}
   \|\(\mathcal{D}_{mn}(x_\mu+x,y_\nu+y)\)\|_{\widetilde{\ell}^{mn}}&=\|\({D}_{mn}(x_\mu+x,y_\nu+y)\)\|_{\widetilde{\ell}^{mn}}
   +\mathcal{O}\(\log n+p \log \frac{m}{p}\)\\
   &=\|\(S_{mn}(x_\mu+x,y_\nu+y)\)\|_{\widetilde{\ell}^{mn}}+\mathcal{O}\(\log n+p \log \frac{m}{p}\).
\end{split}
\end{equation}
Next, by the same arguments as in the proof of Lemma~\ref{le5}, using~\eqref{leKu3} and~\eqref{leKu4}, we get
    \begin{equation}\label{28}
    \begin{split}
  \|\,(S_{mn}(x_\mu+x,&y_\nu+y))\,\|_{\widetilde{\ell}^{mn}}=\frac2{m^2}\sum_{B_m}\Big|\D_m^{(1)}(x+x_\mu,\frac nm(y+y_\nu))\Big|\\
  &+\(\log \frac nm\)\frac{2}{\pi m^2} \sum_{B_m}\Big|\D_m^{(2)}(x+x_\mu,\frac nm(y+y_\nu))\Big|+\mathcal{O}(\log m)\\
&=\frac{2}{\pi^2}\Big|\sin{\frac{mx+ny}{2}}\sin\frac{mx-ny}{2}\Big|(2\log m\log n-\log^2 m)+\mathcal{O}(\log n).
\end{split}
  \end{equation}
Now, combining~\eqref{23}, \eqref{27}, and~\eqref{28}, we obtain
    \begin{equation}\label{29}
    \begin{split}
\Lambda_{mn}^{\rm LC}\le\frac{4}{\pi^2}(2\log m\log n-\log^2 m)+\mathcal{O}\(\log n+p \log \frac{m}{p}\).
\end{split}
  \end{equation}
At the same time, it follows from~\eqref{23} and~\eqref{23+} that
\begin{equation*}
\begin{split}
\Lambda_{mn}^{\rm LC}\ge 2\|\{\mathcal{D}_{mn}(x_{\mu+1},y_\nu)\}\|_{\widetilde{\ell}^{mn}}+\mathcal{O}(\log n).
\end{split}
\end{equation*}
Thus, repeating the above arguments with $(x,y)=(\pi/m,0)$, we get
\begin{equation}\label{31}
\begin{split}
\Lambda_{mn}^{\rm LC}\ge \frac{4}{\pi^2}(2\log m\log n-\log^2 m)+\mathcal{O}\(\log n+p \log \frac{m}{p}\).
\end{split}
\end{equation}

Finally, combining~\eqref{29} and~\eqref{31}, we prove the theorem.
\end{proof}

\begin{proof}[Proof of Theorem~\ref{th1}]
By~\eqref{zv}, we  have $\mathcal{L}_{mn}=\frac1{\pi^2}\|D_{mn}\|_{L([0,\pi)^2)}$.  Using this equality together with~\eqref{le4.1} and~\eqref{le5.1}, we get~\eqref{th1.1}.
Similarly, applying~\eqref{le4.1+} and~\eqref{le5.1}, we prove~\eqref{th1.1c}.
\end{proof}

\begin{proof}[Proof of Theorem~\ref{th1+}]
As in the previous theorem, equality~\eqref{th1.2} follows directly from~\eqref{le4.2} and~\eqref{le5.2}. Here, we use additionally the fact that the Lebesgue function $L_{mn}(x,y)$ is periodic in $x$ with period $\pi/m$ and in $y$ with period $\pi/n$.
Similarly, using~\eqref{le4.2+} and~\eqref{le5.2}, we get~\eqref{th1.2c}.
\end{proof}

%
%
%

%
%
%
%
%

\end{document}